\documentclass[12pt]{amsart}
\usepackage{amsthm}
\usepackage{amsmath}
\usepackage{amssymb}
\usepackage{mathrsfs}
\usepackage{mathtools}
\usepackage{enumerate}
\usepackage{soul}
\usepackage{float}
\usepackage{pgf}
\usepackage{url, multicol}
\usepackage[all]{xy}
\usepackage{graphicx}
\usepackage{hyperref}
\usepackage{physics}
\usepackage{comment}
\usepackage{mathtools}
\usepackage[normalem]{ulem}
\usepackage[utf8]{inputenc}

\newcommand{\Q}{{\mathbb Q}}

\newcommand{\F}{{\mathbb F}}
\newcommand{\N}{{\mathrm{N}}}

\newcommand{\ka}{{\mathfrak a}}

\renewcommand{\ker}{\mathrm{ker}}
\newcommand{\im}{\mathrm{im}}

\renewcommand{\N}{\mathrm{N}}
\newcommand{\kp}{\mathfrak{p}}
\renewcommand{\ka}{\mathfrak{a}}
\newcommand{\km}{\mathfrak{m}}
\newcommand{\re}{\mathrm{Re}}
\renewcommand{\im}{\mathrm{Im}}
\renewcommand{\epsilon}{\varepsilon}

\newtheorem{thm}{Theorem}[section]
\newtheorem{theorem}[thm]{Theorem}
\newtheorem*{theorem*}{Theorem}

\newtheorem{lemma}[thm]{Lemma}

\usepackage[margin=1 in]{geometry} 

\allowdisplaybreaks

\author{Apoorva Panidapu}
\address{Department of Mathematics, San Jose State University, San Jose, CA 95112, USA}
\email{apoorva002@gmail.com}

\author{Jesse Thorner}
\address{Department of Mathematics, University of Illinois, Urbana, IL 61801, USA}
\email{\href{mailto:jesse.thorner@gmail.com}{jesse.thorner@gmail.com}}

\title{Short-interval sector problems for CM elliptic curves}

\begin{document}

\maketitle

\begin{abstract}
Let $E/\Q$ be an elliptic curve that has complex multiplication (CM) by an imaginary quadratic field $K$.  For a prime $p$, there exists $\theta_p \in [0, \pi]$ such that $p+1-\#E(\F_p) = 2\sqrt{p} \cos \theta_p$.  Let $x>0$ be large, and let $I\subseteq[0,\pi]$ be a subinterval.  We prove that if $\delta>0$ and $\theta>0$ are fixed numbers such that $\delta+\theta<\frac{5}{24}$, $x^{1-\delta}\leq h\leq x$, and $|I|\geq x^{-\theta}$, then 
\[
\frac{1}{h}\sum_{\substack{x < p \le x+h \\ \theta_p \in I}}\log{p}\sim \frac{1}{2}\mathbf{1}_{\frac{\pi}{2}\in I}+\frac{|I|}{2\pi},
\]
where $\mathbf{1}_{\frac{\pi}{2}\in I}$ equals 1 if $\frac{\pi}{2}\in I$ and $0$ otherwise.  We also discuss an extension of this result to the distribution of the Fourier coefficients of holomorphic cuspidal CM  newforms.
\end{abstract}

\section{Introduction and Statement of Result}\label{introduction}
 
Let $E:y^2 = x^3+ax+b$ be an elliptic curve over $\Q$ of conductor $N_E$ that has complex multiplication (CM) by an imaginary quadratic field $K$.  Such a field $K$ necessarily has class number 1, so all fractional ideals are principally generated and the ring of integers $\mathcal{O}_K$ of $K$ is a unique factorization domain.  For a prime $p$, let $E(\mathbb{F}_p)$ be the group of $\mathbb{F}_p$-rational points on the reduction of $E$ modulo $p$.  Hasse proved that there exists $\theta_p\in[0,\pi]$ such that
\begin{equation}
\label{eqn:Hasse}
p+1-\#E(\mathbb{F}_p) = 2\sqrt{p}\cos\theta_p.
\end{equation}
A classical result of Deuring states that for primes $p\nmid N_E$, we have that $\theta_p = \frac{\pi}{2}$ if and only if $p$ is inert in $K$.  Such primes have density $\frac{1}{2}$.  For the density $\frac{1}{2}$ subset of the primes $p\nmid D_K$ that split in $K$, the angles $\theta_p$ are equidistributed in $[0,\pi]$.  This follows from work of Hecke; see \cite[Section 7]{hecke} and \cite[Section 2]{Sutherland}.  The following theorem summarizes these results.

\begin{theorem*}
Let $E/\Q$ be a CM elliptic curve.  For a prime $p$, define $\theta_p$ by \eqref{eqn:Hasse}.  For a fixed subinterval $I \subseteq [0, \pi]$, let $\mu(I)=\frac{|I|}{2\pi}$ if $\frac{\pi}{2}\notin I$ and $\mu(I)=\frac{1}{2}+\frac{|I|}{2\pi}$ if $\frac{\pi}{2}\in I$.  We have that
\begin{equation}
\label{eqn:ST_CM}
\lim_{x \to \infty} \frac{1}{x}\sum_{\substack{x<p \le 2x \\ \theta_p \in I}}\log p = \mu(I).
\end{equation}
\end{theorem*}
In this paper, we prove a result that contains several refinements of \eqref{eqn:ST_CM}.

\begin{theorem}
\label{thm:main_theorem}
Let $E/\Q$ be a fixed CM elliptic curve.  Let $I \subseteq [0, \pi]$ be a subinterval, and for a prime $p$, define $\theta_p$ by \eqref{eqn:Hasse}.  Let $\delta>0$ and $\theta>0$ be fixed and satisfy $\delta+\theta<\frac{5}{24}$.  There exists a constant  $c>0$ such that if $x\geq 3$, $x^{1-\delta}\leq h\leq x$, and $|I|\geq x^{-\theta}$, then
\[
\frac{1}{h}\sum_{\substack{x < p \le x+h \\ \theta_p \in I}}\log{p} = \mu(I) \Big(1 + O\Big(\exp\Big(-c\frac{(\log x)^{1/3}}{(\log\log x)^{1/3}}\Big)\Big)\Big).
\]
\end{theorem}

Theorem \ref{thm:main_theorem} implies \eqref{eqn:ST_CM} with much more uniformity, specifying the rate of convergence in the limit and showing that one can take $|I|$ to be small (so that the angles $\theta_p$ exhibit small-scale equidistribution) while localizing the primes themselves to a short interval.  When $I=[0,\pi]$, our work shows that if $0<\delta<\frac{5}{24}$ and $x^{1-\delta}\leq h\leq x$, then $\frac{1}{h}\sum_{x<p\leq x+h}\log p\sim 1$.  This was first proved with $\delta\leq \frac{1}{33000}$ by Hoheisel \cite{hoheisel} and gives strong refinement of the classical prime number theorem $\frac{1}{x}\sum_{x<p\leq 2x}\log p\sim 1$.  In Section 2, we reduce the proof of Theorem \ref{thm:main_theorem} to the study of partial sums for Hecke Gr{\"o}ssencharacters defined over an imaginary quadratic field of class number one.  In Section 3, we use classical ``explicit formulae'' to relate the partial sums to the zeros of Hecke Gr{\"o}ssencharacter $L$-functions, which are studied using a zero-free region and zero density estimate due to Coleman.  In Section 4, we describe how our ideas extend to study the distribution of Fourier coefficients of holomorphic cuspidal Hecke eigenforms which are also newforms with complex multiplication.

\subsection*{Acknowledgements} This research was supported by the NSF (DMS-2002265), the NSA (H98230-20-1-0012), the Templeton World Charity Foundation, and the Thomas Jefferson Fund at the University of Virginia. The authors thank Dr. Ken Ono for his helpful suggestions and conversations. Finally, we thank the anonymous referee for their helpful comments.

\section{Preliminary manipulations}
Let $E:y^2 = x^3+ax+b$ be an elliptic curve defined over $\Q$ with conductor $N_E$.  Suppose that $E/\Q$ has complex multiplication by an imaginary quadratic number field $K$ having ring of integers $\mathcal{O}_K$ and absolute norm $\mathrm{N}=\mathrm{N}_{K/\mathbb{Q}}$ defined by $\mathrm{N}\mathfrak{a}=|\mathcal{O}_K/\mathfrak{a}|$ for a nonzero ideal $\mathfrak{a}$ of $\mathcal{O}_K$.  The field $K$ necessarily has class number one (so any fractional ideal $\ka_1/\ka_2$ with $\ka_1$ and $\ka_2\neq (0)$ ideals of $\mathcal{O}_K$ is principally generated), so $K$ must be of the form $\Q(\sqrt{-d})$ with $d\in\{1, 2, 3, 7, 11, 19, 43, 67, 163\}$.  If $K=\Q(\sqrt{-d})$ with $d$ squarefree and $-d\equiv 1\pmod{4}$, then the absolute discriminant $D_K$ is $|d|$; if $-d\not\equiv 1\pmod{4}$, then $D_K=4|d|$.  The absolute discriminant $D_K$ divides $N_E$.

Let $I\subseteq[0,\pi]$ be a subinterval with indicator function $\mathbf{1}_I(\theta)$, and let $\mathbf{1}_{\frac{\pi}{2}\in I}=1$ if $\frac{\pi}{2}\in I$ and $0$ otherwise.  Without loss of generality, we may suppose that $I$ is closed.  For a prime $p$, let $\theta_p\in[0,\pi]$ be defined by \eqref{eqn:Hasse}.  We study
\begin{equation}
\label{eqn:sec3sum_1}
\sum_{x<p\leq x+h}\mathbf{1}_{I}(\theta_p)\log p
\end{equation}
for $x>100 N_E$ (so that if $p>x$, then $p$ divides neither $N_E$ nor $D_K$).  For now, we only assume that $\sqrt{x}\log x\leq h\leq x$.  We will use the Brun--Titchmarsh bound \cite[(1.12)]{MV_large}
\begin{equation}
    \label{eqn:BT}
    \sum_{X<p\leq X+Y}1 \leq 2 \frac{Y}{\log Y},\qquad X>0,\quad Y>1.
\end{equation}

\subsection{Fourier analysis}

For $p>x$, we have that $\theta_p=\frac{\pi}{2}$ if and only if $p$ is inert in $K$.  Therefore, \eqref{eqn:sec3sum_1} equals
\begin{equation}
\label{eqn:sec3sum_2}
\sum_{\substack{x<p\leq x+h \\ \textup{$p$ splits in $K$}}}\mathbf{1}_{I}(\theta_p)\log p+\mathbf{1}_{\frac{\pi}{2}\in I}\sum_{\substack{x<p\leq x+h \\ \textup{$p$ is inert in $K$}}}\log p.
\end{equation}
Let $\chi_K$ denote the Kronecker character associated to $K$.  If $p>x$, then
\[
\chi_K(p) = \begin{cases}
1&\mbox{if $p$ splits in $K$,}\\
-1&\mbox{if $p$ is inert in $K$.}
\end{cases}
\]
It follows that \eqref{eqn:sec3sum_2} equals
\begin{multline*}
\mu(I)h+\mu(I)\Big(\sum_{x<p\leq x+h}\log p - h\Big)+(\mu(I)-\mathbf{1}_{\frac{\pi}{2}\in I})\sum_{x<p\leq x+h}\chi_K(p)\log p\\
+\Big(\sum_{\substack{x<p\leq x+h \\ \textup{$p$ splits in $K$}}}\mathbf{1}_I(\theta_p)\log p-\frac{|I|}{\pi}\sum_{\substack{x<p\leq x+h \\ \textup{$p$ splits in $K$}}}\log p \Big).
\end{multline*}
Writing
\[
\delta(\chi)=\begin{cases}
1&\mbox{if $\chi=1$,}\\
0&\mbox{if $\chi=\chi_K$,}
\end{cases}
\]
we observe that
\begin{multline*}
\mu(I)\Big(\sum_{x<p\leq x+h}\log p - h\Big)+(\mu(I)-\mathbf{1}_{\frac{\pi}{2}\in I})\sum_{x<p\leq x+h}\chi_K(p)\log p\\
\ll \mu(I)\sum_{\chi\in\{1,\chi_K\}}\Big|\sum_{x<p\leq x+h}\chi(p)\log p - \delta(\chi)h\Big|.
\end{multline*}

The following lemma enables us to address the contribution from the split primes.
\begin{lemma}
	\label{lem:BS}
	Let $M\geq 1$, and let $I\subseteq[0,\pi]$ be a closed subinterval.  There exist trigonometric polynomials
	\[
	S_{M}^-(\theta) = \sum_{0\leq m\leq M} b_m^- \cos(m\theta),\qquad S_{M}^+(\theta) = \sum_{0\leq m\leq M} b_m^+ \cos(m\theta)
	\]
	such that the following properties hold.
	\begin{enumerate}
		\item If $\theta\in[0,\pi]$, then $S_M^-(\theta)\leq\mathbf{1}_{I}(\theta)\leq S_M^+(\theta)$.
		\item We have the bound $b_0^{\pm} = \frac{|I|}{\pi}+O(M^{-1})$.
		\item If $1\leq m\leq M$, then $|b_m^{\pm}|\ll \min\{|I|,m^{-1}\}+M^{-1}$.
	\end{enumerate}
\end{lemma}
\begin{proof}
Write $I=[\alpha,\beta]$ and $J=[\frac{\alpha}{2\pi},\frac{\beta}{2\pi}]$.  In \cite[Lecture 1]{Montgomery}, it is proved that for all $M\geq 1$, there exist trigonometric polynomials
\[
F_M^{-}(\theta) = \sum_{|n|\leq M}B_m^{-}(n) e^{2\pi i n \theta},\qquad F_M^{+}(\theta) = \sum_{|n|\leq M}B_m^{+}(n) e^{2\pi i n \theta}
\]
satisfying the following properties:
\begin{enumerate}
    \item For all $\theta\in[0,1]$, we have $F_{M}^{-}(\theta)\leq \mathbf{1}_{J}(\theta)\leq F_{M}^{+}(\theta)$.
    \item We have that $F_M^+(\frac{\alpha}{2\pi})=F_M^+(\frac{\beta}{2\pi})=1$ and $F_M^-(\frac{\alpha}{2\pi})=F_M^-(\frac{\beta}{2\pi})=0$.
    \item We have $B_M^{\pm}(0)=|J|+O(M^{-1})$.
    \item If $1\leq|n|\leq M$, then $|B_M^{\pm}(n)|\ll \min\{|J|,|n|^{-1}\}+M^{-1}$.
\end{enumerate}
We take $S_{M}^{\pm}(\theta)=F_{M}^{\pm}(\frac{\theta}{2\pi})+F_{M}^{\pm}(-\frac{\theta}{2\pi})$.
\end{proof}
It follows from Lemma \ref{lem:BS} that if $M\in[1,x]$ is an integer, then
{\small\[
\sum_{0\leq m\leq M} b_m^- \sum_{\substack{x<p\leq x+h \\ \textup{$p$ splits in $K$}}} \cos(m\theta_p)\log p \leq \sum_{\substack{x<p\leq x+h  \\ \textup{$p$ splits in $K$}}}\mathbf{1}_{I}(\theta_p)\log p\leq \sum_{0\leq m\leq M} b_m^+ \sum_{\substack{x<p\leq x+h \\ \textup{$p$ splits in $K$}}} \cos(m\theta_p)\log p
\]}%
and
\begin{multline*}
\sum_{\substack{x<p\leq x+h \\ \textup{$p$ splits in $K$}}}\mathbf{1}_{I}(\theta_p)\log p - \frac{|I|}{\pi}\sum_{\substack{x<p\leq x+h \\ \textup{$p$ splits in $K$}}}\log p\\
\ll \frac{1}{M}\sum_{x<p\leq x+h}\log p+\sum_{1\leq m\leq M}\Big(\frac{1}{M}+\min\Big\{|I|,\frac{1}{m}\Big\}\Big)\Big|\sum_{\substack{x<p\leq x+h \\ \textup{$p$ splits in $K$}}}\cos(m\theta_p)\log p\Big|.
\end{multline*}
Using \eqref{eqn:BT} and partial summation, we find that $\sum_{x<p\leq x+h}\log p\ll h$.  Therefore, in summary, if $x> 100 N_E$, $1\leq M\leq x$, and $\sqrt{x}\log x\leq h\leq x$, then
\begin{multline}
\label{eqn:sec3sum_3}
\Big|\sum_{\substack{x<p\leq x+h \\ \theta_p\in I}}\log p - \mu(I)h\Big|\ll \frac{h}{M}+\mu(I)\sum_{\chi\in\{\chi_K,1\}}\Big|\sum_{x<p\leq x+h}\chi(p)\log p-\delta(\chi)h\Big|\\
+\sum_{1\leq m\leq M}\Big(\frac{1}{M}+\min\Big\{|I|,\frac{1}{m}\Big\}\Big)\Big|\sum_{\substack{x<p\leq x+h \\ \textup{$p$ splits in $K$}}}\cos(m\theta_p)\log p\Big|.
\end{multline}

\subsection{Hecke Gr{\"o}ssencharacters}

The Hasse--Weil $L$-function of $E/\Q$ is
\[
L(s,E)=\Big(\prod_{p|N_E}\frac{1}{1-(2\cos\theta_p)p^{-s}}\Big)\prod_{p\nmid N_E}\frac{1}{1-(2\cos\theta_p)p^{-s}+p^{-2s}}.
\]
Convergence is absolute for $\mathrm{Re}(s)>1$.  The completed $L$-function
\[
\Lambda(s,E) = \Big(\frac{\sqrt{N_E}}{2\pi}\Big)^s \Gamma\Big(s+\frac{1}{2}\Big)L(s,E)
\]
is entire of order one, and there exists $W(E)\in\{-1,1\}$ such that for all $s\in\mathbb{C}$, we have that $\Lambda(s,E)=W(E)\Lambda(1-s,E)$.

We now present the definition of a Hecke Gr{\"o}ssencharacter from \cite[Section 12.2]{Iwaniec}.  Since $K$ has class number one, the multiplicative group $I_K$ be the multiplicative group of nonzero fractional ideals $\ka_1/\ka_2$ (where $\ka_1$ and $\ka_2$ are ideals of $\mathcal{O}_K$) equals the group $P_K$ of principal ideals $\{(a)=a\mathcal{O}_K\colon a\in K^{\times}\}$.  Given a nonzero ideal $\km$ of $\mathcal{O}_K$, we denote by $P_{\km}$ the subgroup of $(a)\in P$ such that $a\equiv 1\pmod{\km}$ (that is, $a-1\in\km$).

Given $u\in\mathbb{Z}$, consider the homomorphism $\xi_{\infty}\colon K^{\times}\to \{|z|=1\}$ given by $\xi_{\infty}(a)=(a/|a|)^u$.  In order to extend $\xi_{\infty}$ to a function on ideals, we let $U_K$ denote the unit group of $\mathcal{O}_K$ and choose $\km$ to satisfy  $\ker(\xi_{\infty})=U_{\km}:=\{\eta\in U_K\colon \eta\equiv 1\pmod{\km}\}$.  We may then extend $\xi_{\infty}$ to a character $\xi$ of $P_{\km}$---if $\ka=(a)$ with $a\equiv 1\pmod{\km}$, then $\xi(\ka)=\xi_{\infty}(a)$.

If $\xi\pmod{\km}$ is a Gr{\"o}ssencharacter, then any ideal $\mathfrak{n}$ of $\mathcal{O}_K$ satisfying $\mathfrak{n}\subseteq\km$ is also a modulus of $\xi$..  Moreover, given $\xi\pmod{\km}$, there may exist a Gr{\"o}ssencharacter $\xi^{*}\pmod{\km^{*}}$ with $\km^{*}\subseteq{\km}$ such that $\xi(\ka)=\xi^*(\ka)$ on $P_{\km}$.  The largest ideal $\km^*$ with this property (which is the greatest common divisor of all such ideals) is the conductor of $\xi$.  If $\km=\km^{*}$, then $\xi$ is called primitive.  For any character $\xi\pmod{\km}$, there exists a unique primitive character $\xi^{*}\pmod{\km^{*}}$, such that $\xi(\ka)=\xi^*(\ka)$ if $\gcd(\ka,\km)=\mathcal{O}_K$.  A character $\xi\colon P_{\km}\to\{|z|=1\}$ extends to a (Hecke) Gr{\"o}ssencharacter $\xi\colon P_K\to\mathbb{C}$ by setting $\xi(\ka)=0$ if $\gcd(\ka,\km)\neq\mathcal{O}_K$.

Now, we introduce a crucial result of Deuring relating elliptic curves with complex multiplication to Gr{\"o}ssencharacters.

\begin{theorem}
\label{thm:Grossencharacter}
If $E/\Q$ has CM by an imaginary quadratic field $K$ of absolute discriminant $D$, then there exists an ideal $\km$ of $\mathcal{O}_K$ such that $\N\km=N_E/D_K$ and a primitive Gr{\"o}ssencharacter $\xi\pmod{\km}$ (with $u=1$) such that
\[
L(s,E)=L(s,\xi):=\prod_{\kp}\frac{1}{1-\xi(\kp)\mathrm{N}\kp^{-s}},
\]
where $\kp$ ranges over the prime ideals of $\mathcal{O}_K$.  In particular, if $p\nmid N_E$, then
\[
1-(2\cos\theta_p)p^{-s}+p^{-2s} = \begin{cases}
(1-\xi(\kp)\mathrm{N}\kp^{-s})(1-\overline{\xi(\kp)}\mathrm{N}\kp^{-s})&\mbox{if $p$ splits, $(p) = \kp\bar{\kp}$, $\mathrm{N}\kp=p$,}\\
1-\xi(\kp)\mathrm{N}\kp^{-s}&\mbox{if $p$ is inert, $(p) = \kp$, $\mathrm{N}\kp=p^2$.}
\end{cases}
\]
\end{theorem}
\begin{proof}
See Theorem II.10.5 and Exercises 2.30 and 2.32 (pp. 184-185) in \cite{Silverman}.
\end{proof}

We will also use the $L$-functions attached to powers of $\xi$.

\begin{theorem}
\label{thm:Grossencharacter2}
Let $\xi\pmod{\km}$ be as in Theorem \ref{thm:Grossencharacter}.  Let $\xi_1=\xi$ and $\km_1=\km$.  For each integer $m\geq 2$, let $\xi_m\pmod{\km_m}$ be the primitive Gr{\"o}ssencharacter that induces $\xi^m$.  Define
\[
L(s,\xi_m) = \prod_{\kp}\frac{1}{1-\xi_m(\kp)\N\kp^{-s}},\qquad\mathrm{Re}(s)>1.
\]
For each $m\geq 1$, the following statements hold.
\begin{enumerate}
    \item The function
    \[
    \Lambda(s,\xi_m)=\Big(\frac{\sqrt{D_K\mathrm{N}\km_m}}{2\pi}\Big)^{s}\Gamma\Big(s+\frac{m}{2}\Big)L(s,\xi_m)
    \]
    is entire of order one.
    \item There exists a complex number $W(\xi_m)$ of modulus one such that for all $s\in\mathbb{C}$, we have that $\Lambda(s,\xi_m)=W(\xi_m)\Lambda(1-s,\overline{\xi_m})$.
    \item $\N\km_m$ divides $N_E/D_K$.
    \item For each prime $p$ and each integer $j\geq 2$, there exists a complex number $a_m(p^j)$ such that (i) $|a_m(p^j)|\leq 2$, and (ii) if $\mathrm{Re}(s)>1$, then
    \[
    -\frac{L'(s,\xi_m)}{L(s,\xi_m)} = \sum_{\textup{$p$ splits in $K$}}\frac{2\cos(m\theta_p)\log p}{p^{s}}+\sum_{p}\sum_{j=2}^{\infty}\frac{a_m(p^j)\log p}{p^{js}}.
    \]
\end{enumerate}
\end{theorem}
\begin{proof}
Parts (1) and (2) follow from \cite[Theorem 12.3]{Iwaniec}.  Part (3) follows since $\km$ is a modulus for $\xi^m$.  Part (4) follows from the relationship between $\xi(\kp)$ and $\theta_p$ in Theorem \ref{thm:Grossencharacter}.
\end{proof}

We also have:

\begin{theorem}
\label{thm:Grossencharacter0}
Let $\chi\in \{1,\chi_K\}$. Let $q$ be the conductor of $\chi$. Define
\[
L(s,\chi_K) = \prod_p (1-\chi(p)p^{-s})^{-1},\qquad L(s,1) = \zeta(s) = \prod_p (1-p^{-s})^{-1},
\]
where $\zeta(s)$ denotes the Riemann zeta function.  The Euler product for $L(s,\chi)$ converges absolutely for $\re(s)>1$. Moreover, if
\[
\Lambda(s,\chi) =  (s(s-1))^{\delta(\chi)}q^{s/2}\pi^{-s/2}\Gamma\Big(\frac{s+\frac{1-\chi(-1)}{2}}{2}\Big)L(s,\chi),
\]
then $\Lambda(s,\chi)$ is an entire function of order one, and there exists a complex number $W(\chi)$ of modulus one such that $\Lambda(s,\chi)=W(\chi)\Lambda(1-s,\chi)$.  If $\mathrm{Re}(s)>1$, then
\[
-\frac{L'(s,\chi)}{L(s,\chi)}=\sum_{p}\sum_{j=1}^{\infty}\frac{\chi(p)^j \log p}{p^{js}}.
\]
\end{theorem}
\begin{proof}
See \cite[\S 8-\S 12]{davenport}.
\end{proof}

Consider the von Mangoldt function
\[
\Lambda(n) = \begin{cases}
\log p&\mbox{if there exists a prime $p$ and an integer $j\geq 1$ such that $n=p^j$,}\\
0&\mbox{otherwise.}
\end{cases}
\]
Using Theorem \ref{thm:Grossencharacter2}, \eqref{eqn:BT}, and partial summation, we have that
\begin{align*}
    &\Big|\sum_{\substack{x<p\leq x+h \\ \textup{$p$ splits in $K$}}}\cos(m\theta_p)\log p-\frac{1}{2}\sum_{x<n\leq x+h}a_m(n)\Lambda(n)\Big|\leq \sum_{j=2}^{\infty}\sum_{x<p^j\leq x+h}\log p\ll \frac{h}{\sqrt{x}},\\
    &\Big|\sum_{x<p\leq x+h}\chi(p)\log p-\sum_{x<n\leq x+h}\chi(n)\Lambda(n)\Big|\leq \sum_{j=2}^{\infty}\sum_{x<p^j\leq x+h}\log p\ll \frac{h}{\sqrt{x}}.
\end{align*}
If we assume that $M\leq \sqrt{x}(\log x)^{-1}$, then our work shows that \eqref{eqn:sec3sum_3} is
\begin{multline}
\label{eqn:sec3sum_5}
\ll\frac{h}{M}+ \sum_{1\leq m\leq M}\Big(\frac{1}{M}+\min\Big\{|I|,\frac{1}{m}\Big\}\Big)\Big|\sum_{x<n\leq x+h}a_m(n)\Lambda(n)\Big|\\
+\mu(I)\sum_{\chi\in\{1,\chi_K\}}\Big|\sum_{x<n\leq x+h}\chi(n)\Lambda(n)-\delta(\chi)h\Big|.
\end{multline}


\section{Proof of Theorem \ref{thm:main_theorem}}

The Euler product that defines $L(s,\xi_m)$ converges absolutely for $\mathrm{Re}(s)>1$, and it follows that $L(s,\xi_m)\neq 0$ in that region.  By the functional equation of $\Lambda(s,\xi_m)$, it follows that unless $j\geq 0$ is an integer and $s=-j-\frac{m}{2}$, we have that $L(s,\xi_m)\neq 0$ for $\re(s)<0$.  These trivial zeros arise as poles of $\Gamma(s+\frac{m}{2})$.  Because $\Lambda(s,\xi_m)$ is entire of order one, it exhibits a Hadamard factorization---there exist complex numbers $A_m$ and $B_m$ such that
\[
\Lambda(s,\xi_m)=e^{A_m+B_m s}\prod_{\rho_m}\Big(1-\frac{s}{\rho_m}\Big)e^{\frac{s}{\rho_m}},
\]
where $\rho_m=\beta_m+i\gamma_m$ denotes a nontrivial zero of $L(s,\xi_m)$.  These zeros all satisfy $0\leq\beta_m\leq 1$.  The nontrivial zeros are directly related to the partial sums of $a_m(n)\Lambda(n)$ by means of an ``explicit formula".

\begin{lemma}
	\label{lem:explicit_formula}
	If $x> 100 N_E$, $1\leq m\leq M\leq \sqrt{x}(\log x)^{-1}$, and $\sqrt{x}\log x\leq h\leq x$, then
	\[
	\sum_{x<n\leq x+h}a_m(n)\Lambda(n) = -\sum_{\substack{\rho_m=\beta_m+i\gamma_m \\ |\gamma_m|\leq Mx/h \\ 0\leq\beta_m\leq 1}}\frac{(x+h)^{\rho_m}-x^{\rho_m}}{\rho_m}+O\Big(\frac{h(\log x)^2}{M}\Big).
	\]
\end{lemma}
\begin{proof}
Since at most two prime ideals of $\mathcal{O}_K$ lie above a given rational prime $p\nmid D_K$, we can view $L(s,\xi_m)$ as an $L$-function with an Euler product of degree 2 over $\Q$.  The conductor of this $L$-function is $D_K\mathrm{N}\km_m$ (which, since $E$ is fixed and there are only nine imaginary quadratic fields of class number 1, is $O(1)$ per Theorem \ref{thm:Grossencharacter2}), and the analytic conductor $\mathfrak{q}_m$ (per \cite[Equation 5.7]{Iwaniec-Kowalski}) satisfies $\log \mathfrak{q}_m\asymp \log m\ll \log M\ll \log x$.  By \cite[Equation 5.53]{Iwaniec-Kowalski}, it follows that
	\[
  \sum_{n\leq x}a_m(n)\Lambda(n) = -\sum_{\substack{\rho_m=\beta+m+i\gamma_m \\ |\gamma_m|\leq Mx/h \\ 0\leq\beta_m\leq 1}}\frac{x^{\rho_m}-1}{\rho_m}+O\Big(\frac{h(\log x)^2}{M}\Big)
	\]
	and
	\[
  \sum_{n\leq x+h}a_m(n)\Lambda(n) = -\sum_{\substack{\rho_m=\beta+m+i\gamma_m \\ |\gamma_m|\leq Mx/h \\ 0\leq\beta_m\leq 1}}\frac{(x+h)^{\rho_m}-1}{\rho_m}+O\Big(\frac{h(\log x)^2}{M}\Big).
	\]
	Subtracting the former from the latter, we obtain the desired result.
\end{proof}

The corresponding results for Dirichlet $L$-functions are proved in \cite[\S15-\S17,\S19]{davenport}.

\begin{lemma}
\label{lem:dirichletL-func}
Let $\chi\in\{1,\chi_K\}$.  If $\rho_{\chi}=\beta_{\chi}+i\gamma_{\chi}$ denotes a nontrivial zero of $L(s,\chi)$, $x>100N_E$, $1\leq M\leq \sqrt{x}(\log x)^{-1}$, and $\sqrt{x}\log x\leq h\leq x$, then
\[
\sum_{x<n\leq x+h} \chi(n)\Lambda(n) - \delta(\chi)h = -\sum_{\substack{\rho_{\chi}=\beta_{\chi}+i\gamma_{\chi} \\ |\gamma_{\chi}|\leq Mx/h \\ 0\leq\beta_{\chi}\leq 1}}\frac{(x+h)^{\rho_{\chi}}-x^{\rho_{\chi}}}{\rho_{\chi}}+O\Big(\frac{h(\log x)^2}{M}\Big).
\]
\end{lemma}

\subsection{Reduction to the zeros}
Let $0<\epsilon<\frac{5}{24}$.  Choose $\delta>0$ and $\theta>0$ so that $\delta+\theta\leq\frac{5}{24}-\epsilon$, and let $0<\epsilon_1<\theta$.  Recall that $x>100N_E$, and assume that $|I|\geq x^{\epsilon_1-\theta}$, $M= x^{\theta}$, and $x^{1-\delta}\leq h\leq x$.  Using Lemma \ref{lem:explicit_formula} and Theorem \ref{lem:dirichletL-func}, we find that \eqref{eqn:sec3sum_5} is
\begin{multline}
	\label{eqn:sec4sum_1}
	\ll \sum_{1\leq m\leq M}\Big(\frac{1}{M}+\min\Big\{|I|,\frac{1}{m}\Big\}\Big)\sum_{\substack{\rho_m=\beta_m+i\gamma_m \\ |\gamma_m|\leq Mx/h \\ 0\leq\beta_m\leq 1}}\Big|\frac{(x+h)^{\rho_m}-x^{\rho_m}}{\rho_m}\Big|\\
+\mu(I)\sum_{\chi\in\{1,\chi_K\}}\sum_{\substack{\rho_{\chi}=\beta_{\chi}+i\gamma_{\chi} \\ |\gamma_{\chi}|\leq Mx/h \\ 0\leq\beta_{\chi}\leq 1}}\Big|\frac{(x+h)^{\rho_{\chi}}-x^{\rho_{\chi}}}{\rho_{\chi}}\Big|+\frac{h(\log x)^3}{M}.
\end{multline}
Note that $1/M \ll \mu(I) x^{-\epsilon_1}$.  Therefore, \eqref{eqn:sec4sum_1} is
\begin{equation}
    \label{eqn:sec4sum_2}
 	\ll \mu(I)\Big(\sum_{\substack{\rho=\beta+i\gamma \\ |\gamma|\leq Mx/h \\ 0\leq\beta\leq 1}}\Big|\frac{(x+h)^{\rho}-x^{\rho}}{\rho}\Big|+h x^{-\epsilon_1}(\log x)^3\Big).
\end{equation}
Here, $\rho$ ranges over all nontrivial zeros of
\[
L(s,M):=\zeta(s)L(s,\chi_K)\prod_{1\leq m\leq M}L(s,\xi_m).
\]
Since $\rho=\beta+i\gamma$ with $0\leq\beta\leq 1$ and $\gamma\in\mathbb{R}$, it follows that
\[
\Big|\frac{(x+h)^{\rho}-x^{\rho}}{\rho}\Big|=\Big|\int_{x}^{x+h}t^{\rho-1}dt\Big|\leq \int_x^{x+h}t^{\beta-1}dt\ll hx^{\beta-1},
\]
so \eqref{eqn:sec4sum_2} is
\begin{equation}
	\label{eqn:sec4sum_3}
	\ll_{\epsilon_1} \mu(I)h\Big(\sum_{\substack{\rho=\beta+i\gamma \\ |\gamma|\leq Mx/h \\ 0\leq\beta\leq 1}}x^{\beta-1}+x^{-\epsilon_1/2}\Big).
\end{equation}
Note that $L(s,\overline{\xi_m}) = \overline{L(\overline{s},\xi_m)}$ and $L(s,\bar{\chi})=\overline{L(\bar{s},\chi)}$.  Therefore, if $\beta+i\gamma$ is a nontrivial zero of $L(s,M)$, then so is $1-\beta+i\gamma$.  Consequently, \eqref{eqn:sec4sum_3} is
\begin{equation}
	\label{eqn:sec4sum_4}
	\ll_{\epsilon_1} \mu(I)h\Big(\sum_{\substack{\rho=\beta+i\gamma \\ |\gamma|\leq Mx/h \\ \frac{1}{2}\leq\beta\leq 1}}x^{\beta-1}+x^{-\epsilon_1/2}\Big).
\end{equation}
We now proceed as in the work of Hoheisel \cite{hoheisel}.

\subsection{Zeros of Hecke Gr{\"o}ssencharacter {\it L}-functions}

Define
\[
N(\sigma,M,T):=\#\{\rho=\beta+i\gamma\colon |\gamma|\leq T,~\beta\geq\sigma\},
\]
where $\rho$ runs through the nontrivial zeros of $L(s,M)$ counted with multiplicity.  By partial summation, \eqref{eqn:sec4sum_4} is
\begin{equation}
\label{eqn:sec4sum_5}
\ll_{\epsilon_1} (\mathbf{1}_{\frac{\pi}{2}\in I}+|I|) h\Big((\log x)\max_{\frac{1}{2}\leq \sigma\leq 1}x^{\sigma-1}N(\sigma,M,Mx/h)+x^{-\epsilon_1 / 2}\Big).
\end{equation}
We use the following zero-free region for the $L$-functions in question.
\begin{lemma}
\label{ZFR}
There exists a constant $c > 0$, independent of $M$, such that $L(s,M)\neq 0$ in the region
\[
\re(s) \ge 1 - \frac{c}{(\log (|\im(s)|+3)M)^{2/3}  (\log \log (|\im(s)|+3)M)^{1/3}}.
\]
\end{lemma}
\begin{proof}
This follows from work of Coleman \cite[Theorem 2]{colemanZFR}.
\end{proof}

Lemma \ref{ZFR} implies that \eqref{eqn:sec4sum_5} is
\begin{equation}
	\label{eqn:sec4sum_6}
\ll_{\epsilon_1} \mu(I)h\Big((\log x)\max_{\frac{1}{2}\leq \sigma\leq 1-\frac{c}{(\log M^2 x/h)^{2/3}(\log\log M^2 x/h)^{1/3}}}x^{\sigma-1}N(\sigma,M,Mx/h)+x^{-\epsilon_1 / 2}\Big).\hspace{-1.5mm}
\end{equation}
We now employ a zero density estimate  for Hecke Gr{\"o}ssencharacter $L$-functions due to Coleman \cite[Corollary 8.2]{coleman1990}, verifying that $N(\sigma,M,T)$ becomes very small as $\sigma\to 1^-$.  This serves as a substitute for the yet-unproven generalized Riemann hypothesis for $L(s,M)$:  If $\re(s)>\frac{1}{2}$, then  $L(s,M)\neq 0$.
\begin{lemma}
\label{ZDE}
For all $\epsilon>0$, there exists a constant $B>0$ such that if $M,T\geq 2$, then
\[
N(\sigma,M,T) \ll_{\epsilon} (M^2+T^2)^{(\frac{12}{5}+\epsilon)(1 - \sigma)}(\log{MT})^B,\qquad \tfrac{1}{2}\leq\sigma\leq 1.
\]
\end{lemma}

\subsection{Proof of Theorem \ref{thm:main_theorem}}

We apply Lemma \ref{ZDE} to \eqref{eqn:sec4sum_6} to achieve the bound
\begin{multline*}
	\max_{\frac{1}{2}\leq \sigma\leq 1-\frac{c}{(\log M^2 x/h)^{2/3}(\log\log M^2 x/h)^{1/3}}}x^{\sigma-1}N(\sigma,M,Mx/h)\\
	\ll_{\epsilon} (\log x)^{B}\max_{\frac{1}{2}\leq \sigma\leq 1-\frac{c}{(\log M^2 x/h)^{2/3}(\log\log M^2 x/h)^{1/3}}}\Big(\frac{(Mx/h)^{2(\frac{12}{5}+\epsilon)}}{x}\Big)^{1-\sigma}.
\end{multline*}
Our constraints on $h$, $M$, $\delta$, and $\theta$ imply that the above display is
\[
\ll_{\epsilon} (\log x)^{B}\max_{\frac{1}{2}\leq \sigma\leq 1-\frac{c}{2(\log x)^{2/3}(\log\log x)^{1/3}}}x^{-\frac{263}{60}\epsilon(1-\sigma)}\ll_{\epsilon}(\log x)^B\exp\Big(-\frac{263c\epsilon}{120}\frac{(\log x)^{1/3}}{(\log\log x)^{1/3}}\Big).
\]
If $\epsilon_1$ is chosen suitably in terms of $\epsilon$ and $x$ is suitably large with respect to $\epsilon$, then \eqref{eqn:sec4sum_6} is
\begin{equation}
\label{eqn:sec4sum7}
\ll_{\epsilon}\mu(I)h \exp\Big(-\frac{263c\epsilon}{240}\frac{(\log x)^{1/3}}{(\log\log x)^{1/3}}\Big).
\end{equation}
Theorem \ref{thm:main_theorem} now follows from \eqref{eqn:sec3sum_2}-\eqref{eqn:sec4sum7}.

\section{A further extension}

We now describe a natural extension of our work to the study of Fourier coefficients of holomorphic cuspidal newforms.  Let
\[
f(z) = \sum_{n\geq 1}\lambda_f(n) n^{\frac{k-1}{2}}e^{2\pi inz}
\]
be the Fourier series of a holomorphic cusp form of even integer weight $k\geq 2$ that transforms under a congruence subgroup $\Gamma_0(N_f)\subseteq \mathrm{SL}_2(\mathbb{Z})$, where $N_f\geq 1$ is the level of $f$.  Suppose further that $f$ is a newform, so that it is an eigenform of all of the Hecke operators.

The $L$-function $L(s,f)$ associated to $f$ is given by
\[
L(s,f) = \prod_{p}\frac{1}{1-\lambda_f(p)p^{-s}+\chi(p)p^{-2s}},\qquad \mathrm{Re}(s)>1,
\]
where $\chi\pmod{N_f}$ is the nebentypus character of $f$.  Suppose that $f$ has complex multiplication by an imaginary quadratic field $K$.  Equivalently, for $p\nmid N_f$, we have that $\lambda_f(p)=0$ if and only if $p$ is inert in $K$.  Unlike the case for elliptic curves, it could be the case that $K$ has a class number larger than 1, so the ideal class group $\mathrm{Cl}_K$ given by the abelian quotient $I_K/P_K$ is no longer a trivial group.  The work of Deuring can be extended to show that there exists a primitive Gr{\"o}ssencharacter $\xi\pmod{\km}$ defined over $K$ with $u=k-1$ and  $\mathrm{N}\km=N_f/D_K$ such that $L(s,f) = L(s,\xi)$. (See \cite[Section 12.3]{Iwaniec}.)  It follows that there exists $\theta_p\in[0,\pi]$ such that $\lambda_f(p) = 2\cos\theta_p$, and for $p\nmid N$, we have $\theta_p=\frac{\pi}{2}$ if and only if $p$ is inert in $K$.  We can now ask how the angles $\theta_p$ are distributed in the interval $[0,\pi]$.  In order to relate the values $\cos(m\theta_p)$ to the $L$-functions $L(s,\xi^m)$ like we did for elliptic curves, we restrict our consideration to the primes $p$ such that $(p)$ is principally generated.  In order to isolate such primes, we use the orthogonality of the characters $\psi$ of $\mathrm{Cl}_K$.  One can then relate the prime counting function
\[
\sum_{\substack{x<p\leq x+h\\ \theta_p\in I \\ \textup{$(p)$ principally generated}}}\log p
\]
to the analytic properties of the $L$-functions of the primitive characters that induce $\xi^m \psi$.

Ultimately, the distribution of primes we are interested in reduces to an understanding of the distribution of zeros of these twisted $L$-functions.  Lemmata \ref{ZFR} and \ref{ZDE} were proved in this level of generality in \cite{coleman1990,colemanZFR}, so our proofs will carry through with the aforementioned adjustments so that if $f$ is fixed, $\delta>0$ and $\theta>0$ are fixed and satisfy $\delta+\theta<\frac{5}{24}$, $x^{1-\delta}\leq h\leq x$, and $|I|\geq x^{-\theta}$, and $x$ is suitably large, then
\[
\frac{1}{h}\sum_{\substack{x<p\leq x+h \\ \theta_p\in I \\ \textup{$(p)$ principally generated}}}\log p = \frac{1}{|\mathrm{Cl}_K|}\Big(\frac{1}{2}\mathbf{1}_{\frac{\pi}{2}\in I}+\frac{|I|}{2\pi}\Big) \Big(1+O\Big(\exp\Big(-c'\frac{(\log x)^{1/3}}{(\log\log x)^{1/3}}\Big)\Big)\Big),
\]
where $c'>0$ is a certain constant.  Such a result contains Theorem \ref{thm:main_theorem} as a corollary via the modularity of elliptic curves, namely that for every elliptic curve $E/\mathbb{Q}$ of conductor $N_E$, there exists a weight 2 holomorphic cuspidal newform $f_E$ of level $N_E$ with trivial nebentypus character such that $L(s,E)=L(s,f_E)$.  When $E/\Q$ has CM by an imaginary quadratic field, $f_E$ is explicitly given in \cite[Theorem 12.5]{Iwaniec} via the Gr{\"o}ssencharacter $\xi$ such that $L(s,E)=L(s,\xi)$.

\bibliographystyle{acm} 
\bibliography{biblio} 

\end{document}